\documentclass[12pt,reqno,oneside]{amsart}
\usepackage{amsmath, amssymb}
\usepackage{hyperref}

\numberwithin{equation}{section} 
\newtheorem{theorem}{Theorem}[section]

\newtheorem{lemma}[theorem]{Lemma}
\theoremstyle{definition}
\newtheorem{definition}[theorem]{Definition}
\theoremstyle{remark}
\newtheorem{remark}[theorem]{Remark}
\numberwithin{equation}{section}

\input{tcilatex}

\begin{document}
\title[Meromorphic bi-univalent functions]{Initial coefficient bounds for
certain classes of Meromorphic bi-univalent functions}
\author{H. Orhan, N. Magesh and V.K.Balaji}
\address{Department of Mathematics, \\
Faculty of Science, Ataturk University, \\
25240 Erzurum, Turkey.\\
\texttt{e-mail:} $orhanhalit607@gmail.com$}
\address{ Post-Graduate and Research Department of Mathematics,\\
Government Arts College for Men,\\
Krishnagiri 635001, Tamilnadu, India\\
\texttt{e-mail:} $nmagi\_2000@yahoo.co.in$}
\address{Department of Mathematics, L.N. Govt College, \\
Ponneri, Chennai, Tamilnadu, India.\\
\texttt{e-mail:} $balajilsp@yahoo.co.in$}
\maketitle

\begin{abstract}
In this paper we extend the concept of bi-univalent to the class of
meromorphic functions. We propose to investigate the coefficient estimates
for two classes of meromorphic bi-univalent functions. Also, we find
estimates on the coefficients $|b_0|$ and $|b_1|$ for functions in these new
classes. Some interesting remarks and applications of the results presented
here are also discussed.\ \newline
\end{abstract}


\section{Introduction}

Let $\mathcal{A}$ denote the class of functions of the form 
\begin{equation}  \label{Int-e1}
h(z)=z+\sum\limits_{n=2}^{\infty}a_nz^n
\end{equation}
which are analytic in the open unit disc $\mathbb{U}=\{z: z \in \mathbb{C}%
\,\, \mathrm{and}\,\, |z|<1 \}.$ Further, by $\mathcal{S}$ we shall denote
the class of all functions in $\mathcal{A}$ which are univalent in $\mathbb{U%
}.$ 

It is well known that every function $h\in \mathcal{S}$ has an inverse $%
h^{-1},$ defined by 
\begin{equation*}
h^{-1}(h(z))=z,\,\, (z \in \mathbb{U})
\end{equation*}
and 
\begin{equation*}
h(h^{-1}(w))=w, \,\, (|w| < r_0(h);\,\, r_0(h) \geq \frac{1}{4}),
\end{equation*}
where 
\begin{equation}  \label{Int-f-inver}
h^{-1}(w) = w - a_2w^2 + (2a_2^2-a_3)w^3 - (5a_2^3-5a_2a_3+a_4)w^4+\ldots .
\end{equation}

A function $h \in \mathcal{S}$ is said to be bi-univalent in $\mathbb{U}$ if
both $h(z)$ and $h^{-1}(z)$ are univalent in $\mathbb{U}.$ Let $\Sigma_{%
\mathcal{B}}$ denote the class of bi-univalent functions in $\mathbb{U}$
given by (\ref{Int-e1}).

In 1967, Lewin \cite{Lewin} investigated the bi-univalent function class $%
\Sigma$ and showed that $|a_2|<1.51.$ On the other hand, Brannan and Clunie 
\cite{Bran-1979} (see also \cite{Branna1970,Bran1985,Taha1981}) and
Netanyahu \cite{Netany} made an attempt to introduce various subclasses of
the bi-univalent function class $\Sigma_{\mathcal{B}}$ and obtained
non-sharp coefficient estimates on the first two coefficients $|a_2|$ and $%
|a_3|$ of (\ref{Int-e1}). But the coefficient problem for each of the
following Taylor-Maclaurin coefficients $|a_n|\, (n\in\mathbb{N}%
\setminus\{1,2\};\;\;\mathbb{N}:=\{1,2,3,\cdots\})$ is still an open
problem. Following Brannan and Taha \cite{Bran1985}, many researchers (see%
\cite%
{Ali-Ravi-Ma-Mina-class,Caglar-Orhan,BAF-MKA,Goyal-Goswami,haya,HO-NM-VKB,SSS-VR-VR,HMS-AKM-PG,Xu-HMS-AML,Xu-HMS-AMC}%
) have recently introduced and investigated several interesting subclasses
of the bi-univalent function class $\Sigma_{\mathcal{B}}$ and they have
found non-sharp estimates on the first two Taylor-Maclaurin coefficients $%
|a_2|$ and $|a_3|.$ 

Let $\Sigma$ denote the class of functions $f$ of the form 
\begin{equation}  \label{Int-mero-e1}
f(z)=z+\sum\limits_{n=0}^{\infty}\frac{b_n}{z^n},
\end{equation}
which are mermorphic univalent functions defined in 
\begin{equation*}
\mathcal{V}:=\{z: z\in \mathbb{C} \,\, \mathrm{and}\,\, 1<|z|<\infty\} .
\end{equation*}
It is well known that every function $f\in\Sigma$ has an inverse $f^{-1},$
defined by 
\begin{equation*}
f^{-1}(f(z))=z \qquad(z\in \mathcal{V})
\end{equation*}
and 
\begin{equation*}
f^{-1}(f(w))=w \qquad(M<|w|<\infty,\, M >0).
\end{equation*}
Furthermore, the inverse function $f^{-1}$ has a series expansion of the
form 
\begin{equation}  \label{Int-mero-f-inverse-e1}
f^{-1}(w)= w+\sum\limits_{n=0}^{\infty}\frac{B_n}{w^n},
\end{equation}
where $M<|w|<\infty.$ 

The coefficient problem was investigated for various interesting subclasses
of the meromorphic univalent functions (see, for example \cite%
{PLD-1971,Krzyz-1979,Shiffer-38}). In 1951, Springer \cite{Springer}
conjectured on the coefficient of the inverse of meromorphic univalent
functions, latter the problem was investigated by many researchers for
various subclasses (see, for details \cite%
{PK-HMS-AKM,Krzyz-1979,Kubota,Schober,Xu-Lv-AMC-2013}). 

Analogous to the bi-univalent analytic functions, a function $f\in \Sigma$
is said to be meromorphic bi-univalent if both $f$ and $f^{-1}$ are
meromorphic univalent in $\mathcal{V}.$ We denote by $\Sigma_{\mathcal{M}}$
the class of all meromorphic bi-univalent functions in $\mathcal{V}$ given
by (\ref{Int-mero-e1}). 

A function $f$ in the class $\Sigma$ is said to be meromorphic bi-univalent
starlike of order $\alpha (0\leq \alpha <1)$ if it satisfies the following
inequalities 
\begin{equation*}
f \in \Sigma_{\mathcal{M}}, \,\, \Re \left ( \frac{zf^{\prime }(z)}{f(z)}
\right ) > \alpha\,\,(z\in \mathcal{V}) \,\, \mathrm{and }\,\, \Re \left ( 
\frac{wg^{\prime }(w)}{g(w)} \right )>\alpha\,\,(w\in \mathcal{V}),
\end{equation*}
where $g(w)=f^{-1}(w)$ is the inverse of $f(z)$ whose series expansion is
given by (\ref{Int-mero-f-inverse-e1}), a simple calculation shows that 
\begin{equation}  \label{g-e-mero}
g(w)=w-b_0-\frac{b_1}{w}-\frac{b_2+b_0b_1}{w^2}-\frac{%
b_3+2b_0b_2+b_0^2b_1+b_1^2}{w^3}+\dots .
\end{equation}

We denote by ${\Sigma}^*_{\mathcal{M}}(\alpha)$ the class of all meromorphic
bi-univalent starlike functions of order $\alpha .$ Similarly, a function $f$
in the class $\Sigma$ is said to be meromorphic bi-univalent strongly
starlike of order $\alpha(0<\alpha\leq 1)$ if it satisfies the following
conditions 
\begin{equation*}
f \in \Sigma_{\mathcal{M}}, \,\, \left |\arg \left ( \frac{zf^{\prime }(z)}{%
f(z)} \right )\right | < \frac{\alpha \pi}{2}\,\, (z\in \mathcal{V})\, 
\mathrm{and}\,\, \left |\arg \left ( \frac{wg^{\prime }(w)}{g(w)} \right
)\right | < \frac{\alpha \pi}{2} \,\, (w\in \mathcal{V}),
\end{equation*}
where $g(w)$ is given by (\ref{g-e-mero}). We denote by $\widetilde{\Sigma}_{%
\mathcal{M}}^*(\alpha)$ the class of all meromorphic bi-univalent strongly
starlike functions of order $\alpha .$ The classes $\Sigma_{\mathcal{M}%
}^*(\alpha)$ and $\widetilde{\Sigma}_{\mathcal{M}}^*(\alpha)$ were
introduced and studied by Halim et al. \cite{Halim-Mero-Bi}. 

Motivated by the works of Halim et al. \cite{Halim-Mero-Bi} we define the
following general subclasses $\Sigma^*_{\mathcal{M}}(\alpha, \mu, \lambda)$
and $\widetilde{\Sigma}^*_{\mathcal{M}}(\alpha, \mu, \lambda)$ of the
function class $\Sigma .$

\begin{definition}
\label{Defi-2} A function $f$ given by (\ref{Int-mero-e1}) is said to be in
the class $\Sigma _{\mathcal{M}}^{\ast }(\alpha ,\mu ,\lambda )$ if the
following conditions are satisfied: 
\begin{equation}
f\in \Sigma _{\mathcal{M}},\,\,\Re \left( (1-\lambda )\left( \frac{f(z)}{z}%
\right) ^{\mu }+\lambda f^{\prime }(z)\left( \frac{f(z)}{z}\right) ^{\mu
-1}\right) >\alpha \,\,(\mu \geq 0,\,\lambda \geq 1,\text{ }\lambda >\mu
;z\in \mathcal{V})  \label{Defi-2-e1}
\end{equation}%
and 
\begin{equation}
\Re \left( (1-\lambda )\left( \frac{g(w)}{w}\right) ^{\mu }+\lambda
g^{\prime }(w)\left( \frac{g(w)}{w}\right) ^{\mu -1}\right) >\alpha \,\,(\mu
\geq 0,\,\lambda \geq 1,\text{ }\lambda >\mu ;w\in \mathcal{V})
\label{Defi-2-e2}
\end{equation}%
for some $\alpha (0\leq \alpha <1),$ where $g$ is given by (\ref{g-e-mero}).
\end{definition}

\begin{definition}
\label{Defi-1-e1} A function $f$ given by (\ref{Int-mero-e1}) is said to be
in the class $\widetilde{\Sigma }_{\mathcal{M}}^{\ast }(\alpha ,\mu ,\lambda
)$ if the following conditions are satisfied: 
\begin{equation}
f\in \Sigma _{\mathcal{M}},\,\,\left\vert \arg \left( (1-\lambda )\left( 
\frac{f(z)}{z}\right) ^{\mu }+\lambda f^{\prime }(z)\left( \frac{f(z)}{z}%
\right) ^{\mu -1}\right) \right\vert <\frac{\alpha \pi }{2}\,\,(\mu \geq
0,\,\lambda \geq 1,\text{ }\lambda >\mu ;z\in \mathcal{V})\,
\label{Defi-1-e1}
\end{equation}%
and 
\begin{equation}
\left\vert \arg \left( (1-\lambda )\left( \frac{g(w)}{w}\right) ^{\mu
}+\lambda g^{\prime }(w)\left( \frac{g(w)}{w}\right) ^{\mu -1}\right)
\right\vert <\frac{\alpha \pi }{2}\,\,(\mu \geq 0,\,\lambda \geq 1,\text{ }%
\lambda >\mu ;w\in \mathcal{V})  \label{Defi-1-e2}
\end{equation}%
for some $\alpha (0<\alpha \leq 1),$ where $g$ is given by (\ref{g-e-mero}).
\end{definition}

It is interesting to note that, for $\lambda =1$ and $\mu =0$ the classes $%
\Sigma _{\mathcal{M}}^{\ast }(\alpha ,\mu ,\lambda )$ and $\widetilde{\Sigma 
}_{\mathcal{M}}^{\ast }(\alpha ,\mu ,\lambda )$ respectively, reduces to the
classes $\Sigma _{\mathcal{M}}^{\ast }(\alpha )$ and $\widetilde{\Sigma }_{%
\mathcal{M}}^{\ast }(\alpha )$ introduced and studied by Halim et al. \cite%
{Halim-Mero-Bi}.

The object of the present paper is to extend the concept of bi-univalent to
the class of meromorphic functions defined on $\mathcal{V}$ and find
estimates on the coefficients $|b_0|$ and $|b_1|$ for functions in the
above-defined classes $\Sigma^*_{\mathcal{M}}(\alpha, \mu, \lambda)$ and $%
\widetilde{\Sigma}^*_{\mathcal{M}}(\alpha, \mu, \lambda)$ of the function
class $\Sigma_{\mathcal{M}}$ by employing the techniques used earlier by
Halim et al. \cite{Halim-Mero-Bi}.

In order to derive our main results, we shall need the following lemma.

\begin{lemma}
\textrm{(see \cite{Pom})}\label{lem-pom} If $\varphi\in \mathcal{P},$ then $%
|c_k|\leqq 2$ for each $k,$ where $\mathcal{P}$ is the family of all
functions $\varphi ,$ analytic in $\mathbb{U},$ for which 
\begin{equation*}
\Re\{\varphi(z)\}>0\qquad (z \in \mathbb{U}),
\end{equation*}
where 
\begin{equation*}
\varphi(z)=1+c_1z+c_2z^2+\cdots \qquad (z \in \mathbb{U}).
\end{equation*}
\end{lemma}

\section{Coefficient Bounds for the Function Classes $\Sigma^*_{\mathcal{M}}(%
\protect\alpha, \protect\mu, \protect\lambda)$ and $\widetilde{\Sigma}^*_{%
\mathcal{M}}(\protect\alpha, \protect\mu, \protect\lambda)$}

\ 

We begin this section by finding the estimates on the coefficients $|b_0|$
and $|b_1|$ for functions in the class $\Sigma^*_{\mathcal{M}}(\alpha, \mu,
\lambda).$

\begin{theorem}
\label{Bi-th2} Let the function $f(z)$ given by $(\ref{Int-mero-e1})$ be in
the following class$:$ 
\begin{equation*}
\Sigma _{\mathcal{M}}^{\ast }(\alpha ,\mu ,\lambda )\qquad (0\leq \alpha
<1;\;\lambda \geq 1;\;\;\mu \geq 0;\text{ \ }\lambda >\mu ).
\end{equation*}%
Then 
\begin{equation}
|b_{0}|\leq \frac{2(1-\alpha )}{\lambda -\mu }  \label{bi-th2-b-a2}
\end{equation}%
and 
\begin{equation}
|b_{1}|\leq 2(1-\alpha )\sqrt{\frac{(1-\mu )^{2}(1-\alpha )^{2}}{(\lambda
-\mu )^{4}}+\frac{1}{(2\lambda -\mu )^{2}}}~.  \label{bi-th2-b-a3}
\end{equation}
\end{theorem}

\begin{proof}
It follows from (\ref{Defi-2-e1}) and (\ref{Defi-2-e2}) that 
\begin{equation}  \label{bi-th2-pr-e1}
(1-\lambda)\left(\frac{f(z)}{z}\right)^{\mu}+\lambda f^{\prime }(z)\left(%
\frac{f(z)}{z}\right)^{\mu-1} = \alpha + (1-\alpha)p(z)
\end{equation}
and 
\begin{equation}  \label{bi-th2-pr-e2}
(1-\lambda)\left(\frac{g(w)}{w}\right)^{\mu}+\lambda g^{\prime }(w)\left(%
\frac{g(w)}{w}\right)^{\mu-1} = \alpha + (1-\alpha)q(w),
\end{equation}
where $p(z)$ and $q(w)$ are functions with positive real part in $\mathcal{V}
$ and have the following forms: 
\begin{equation}  \label{Exp-p(z)}
p(z)=1+\frac{p_1}{z}+\frac{p_2}{z^2}+\cdots
\end{equation}
and 
\begin{equation}  \label{Exp-q(w)}
q(z)=1+\frac{q_1}{w}+\frac{q_2}{w^2}+\cdots ,
\end{equation}
respectively. Now, equating coefficients in (\ref{bi-th2-pr-e1}) and (\ref%
{bi-th2-pr-e2}), we get 
\begin{equation}  \label{th2-ceof-p1}
(\mu-\lambda)b_0 = (1-\alpha) p_1,
\end{equation}
\begin{equation}  \label{th2-ceof-p2}
(\mu-2\lambda)(b_1+(\mu-1)\frac{b_0^2}{2}) = (1-\alpha)p_2,
\end{equation}
\begin{equation}  \label{th2-ceof-q1}
(\lambda-\mu)b_0 = (1-\alpha) q_1
\end{equation}
and 
\begin{equation}  \label{th2-ceof-q2}
(2\lambda-\mu)(b_1-(\mu-1)\frac{b_0^2}{2}) = (1-\alpha)q_2.
\end{equation}
From (\ref{th2-ceof-p1}) and (\ref{th2-ceof-q1}), we get 
\begin{equation}  \label{th2-pr-p1=q1}
p_1=-q_1
\end{equation}
and 
\begin{equation}  \label{th2-b-0-square}
b_0^2=\frac{(1-\alpha)^2(p_1^2+q_1^2)}{2(\lambda-\mu)^2}.
\end{equation}
Since $\Re \{p(z)\} > 0$ in $\mathcal{V},$ the function $p(1/z) \in \mathcal{%
P}$ and hence the coefficients $p_n$ and similarly the coefficients $q_n$ of
the function $q$ satisfy the inequality in Lemma \ref{lem-pom}, we get 
\begin{equation*}
|b_0|\leq \frac{2-2\alpha}{\lambda-\mu}.
\end{equation*}
This gives the bound on $|b_0|$ as asserted in (\ref{bi-th2-b-a2}).

Next, in order to find the bound on $|b_1|$, we use (\ref{th2-ceof-p2}) and (%
\ref{th2-ceof-q2}), which yields, 
\begin{equation}  \label{th2-a3-cal-e1}
(1-\mu)^2(2\lambda-\mu)^2b_0^4-4(1-\alpha)^2p_2q_2=4(2\lambda-\mu)^2b_1^2.
\end{equation}
It follows from (\ref{th2-a3-cal-e1}) that 
\begin{align*}
b_1^2 = \frac{(1-\mu)^2b_0^4}{4}-\frac{(1-\alpha)^2}{(2\lambda-\mu)^2}p_2q_2.
\end{align*}
Substituting the estimate obtained (\ref{th2-b-0-square}), and applying
Lemma \ref{lem-pom} once again for the coefficients $p_2$ and $q_2,$ we
readily get 
\begin{equation*}
|b_1| \leq 2(1-\alpha)\sqrt{\frac{(1-\mu)^2(1-\alpha)^2}{(\lambda-\mu)^4}+%
\frac{1}{(2\lambda-\mu)^2}}~.
\end{equation*}
This completes the proof of Theorem \ref{Bi-th2}.
\end{proof}

Next we the estimate the coefficients $|b_0|$ and $|b_1|$ for functions in
the class $\widetilde{\Sigma}^*_{\mathcal{M}}(\alpha, \mu, \lambda) .$

\begin{theorem}
\label{Bi-th1} Let the function $f(z)$ given by $(\ref{Int-e1})$ be in the
following class$:$ 
\begin{equation*}
\widetilde{\Sigma }_{\mathcal{M}}^{\ast }(\alpha ,\mu ,\lambda )\qquad
(0<\alpha \leq 1;\;\;\lambda \geq 1;\;\;\mu \geq 0;\text{ \ }\lambda >\mu ).
\end{equation*}%
Then 
\begin{equation}
|b_{0}|\leq \frac{2\alpha }{\lambda -\mu }  \label{bi-th1-b-a2}
\end{equation}%
and 
\begin{equation}
|b_{1}|\leq 2\alpha ^{2}\sqrt{\frac{1}{(2\lambda -\mu )^{2}}+\frac{(1-\mu
)^{2}}{(\lambda -\mu )^{4}}}.  \label{bi-th1-b-a3}
\end{equation}
\end{theorem}

\begin{proof}
It follows from (\ref{Defi-1-e1}) and (\ref{Defi-1-e2}) that 
\begin{equation}  \label{bi-th1-pr-e1}
(1-\lambda)\left(\frac{f(z)}{z}\right)^{\mu}+\lambda f^{\prime }(z)\left(%
\frac{f(z)}{z}\right)^{\mu-1} = [p(z)]^{\alpha}
\end{equation}
and 
\begin{equation}  \label{bi-th1-pr-e2}
(1-\lambda)\left(\frac{g(w)}{w}\right)^{\mu}+\lambda g^{\prime }(w)\left(%
\frac{g(w)}{w}\right)^{\mu-1} = [q(w)]^{\alpha},
\end{equation}
where $p(z)$ and $q(w)$ have the forms (\ref{Exp-p(z)}) and (\ref{Exp-q(w)}%
), respectively. Now, equating the coefficients in (\ref{bi-th1-pr-e1}) and (%
\ref{bi-th1-pr-e2}), we get 
\begin{equation}  \label{th1-ceof-p1}
(\mu-\lambda)b_0 = \alpha p_1,
\end{equation}
\begin{equation}  \label{th1-ceof-p2}
(\mu-2\lambda)(b_1+(\mu-1)\frac{b_0^2}{2}) = \frac{1}{2}\left
[\alpha(\alpha-1)p_1^2 + 2\alpha p_2\right ],
\end{equation}
\begin{equation}  \label{th1-ceof-q1}
-(\lambda-\mu)b_0 = \alpha q_1
\end{equation}
and 
\begin{equation}  \label{th1-ceof-q2}
(2\lambda-\mu)(b_1-(\mu-1)\frac{b_0^2}{2}) = \frac{1}{2}\left
[\alpha(\alpha-1)q_1^2 + 2\alpha q_2\right ].
\end{equation}
From (\ref{th1-ceof-p1}) and (\ref{th1-ceof-q1}), we find that 
\begin{equation}  \label{th1-pr-p1=q1}
p_1=-q_1
\end{equation}
and 
\begin{equation}  \label{th1-b-0-square}
b_0^2=\frac{\alpha^2(p_1^2+q_1^2)}{2(\lambda-\mu)^2}.
\end{equation}
As discussed in the proof of Theorem \ref{Bi-th2}, applying Lemma \ref%
{lem-pom} for the coefficients $p_2$ and $q_2,$ we immediately have 
\begin{equation*}
|b_0|\leq \frac{2\alpha}{\lambda - \mu}.
\end{equation*}
This gives the bound on $|b_0|$ as asserted in (\ref{bi-th1-b-a2}).

Next, in order to find the bound on $|b_1|$, by using (\ref{th1-ceof-p2})
and (\ref{th1-ceof-q2}), we get 
\begin{equation}  \label{th1-a3-cal-e1}
2(2\lambda-\mu)^2b_1^2+(2\lambda-\mu)^2(1-\mu)^2\frac{b_0^4}{2} = \frac{%
\alpha^2(\alpha-1)^2(p_1^4+q_1^4)}{4}+\alpha^2(p_2^2+q_2^2)+\alpha^2(%
\alpha-1)(p_1^2p_2+q_1^2q_2).
\end{equation}
It follows from (\ref{th1-a3-cal-e1}) and (\ref{th1-b-0-square}) that 
\begin{align*}
2(2\lambda-\mu)^2b_1^2 &= \frac{\alpha^2(\alpha-1)^2(p_1^4+q_1^4)}{4}%
+\alpha^2(p_2^2+q_2^2)+\alpha^2(\alpha-1)(p_1^2p_2+q_1^2q_2) \\
& \qquad- \frac{(2\lambda-\mu)^2(1-\mu)^2\alpha^4}{8(\mu-\lambda)^4}%
(p_1^2+q_1^2)^2.
\end{align*}
Applying Lemma \ref{lem-pom} once again for the coefficients $p_1,$ $p_2,$ $%
q_1$ and $q_2,$ we readily get 
\begin{equation*}
|b_1| \leq 2\alpha^2 \sqrt{\frac{1}{(2\lambda-\mu)^2}+\frac{(1-\mu)^2}{%
(\lambda-\mu)^4}}.
\end{equation*}
This completes the proof of Theorem \ref{Bi-th1}.
\end{proof}

\begin{remark}
For $\lambda=1$ and $\mu=0$ the bounds obtained in Theorems \ref{Bi-th2} and %
\ref{Bi-th1} are coincidence with outcome of \cite[Theorem 1 and Theorem 2]%
{Halim-Mero-Bi}. Similarly, various interesting corollaries and consequences
could be derived from our results, the details involved may be left to the
reader.
\end{remark}


\end{document}